\documentclass[12pt]{amsart} 
\usepackage{color}
\usepackage{amssymb}
\usepackage{times}
\usepackage{amsmath}
\usepackage[pdftex]{graphicx}
\usepackage{hyperref}
\newcounter{sideremark}

\newtheorem{Thm}{Theorem}[section]
\newtheorem{Lem}[Thm]{Lemma}
\newtheorem{Prop}[Thm]{Proposition}
   
\newtheorem{Def}[Thm]{Definition}

\renewcommand{\marginpar}[1]{}
\catcode`\@=12

\def\Empty{}
\newcommand\oplabel[1]{
  \def\OpArg{#1} \ifx \OpArg\Empty {} \else
  	\label{#1}
  \fi}
		
\long\def\realfig#1#2#3#4{
\begin{figure}
  \centering
  \includegraphics{#2}
  \caption[#1]{#3}
  \label{#1}
\end{figure}}

\newcommand{\comm}[1]{}

\renewcommand{\epsilon}{\varepsilon}

\renewcommand{\rho}{\varrho}

\begin{document}

\title{Constructing monotone homotopies and sweepouts}

\author[E. Chambers]{Erin Wolf Chambers}
\address{Department of Computer Science, Saint Louis University, 220 N. Grand Ave., Saint Louis, MO, USA}
\email{echambe5@slu.edu}
\author[G.R. Chambers]{Gregory R. Chambers}
\address{Department of Mathematics, Rice University, 6100 Main St., Houston, TX, USA}
\email{gchambers@rice.edu}
\author[A. de Mesmay]{Arnaud de Mesmay}
\address{LIGM, CNRS, Univ. Gustave Eiffel, ESIEE Paris, F-77454 Marne-la-Vall\'{e}e, France}
\email{arnaud.de-mesmay@univ-eiffel.fr}
\author[T. Ophelders]{Tim Ophelders}
\address{Department of Mathematics and Computer Science, TU Eindhoven, The Netherlands}
\email{t.a.e.ophelders@tue.nl}
\author[R. Rotman]{Regina Rotman}
\address{Department of Mathematics, University of Toronto, 40 St. George Street, Toronto, ON M5S 2E4, Canada}
\email{rina@math.toronto.edu}

\date{\today}
\begin{abstract}
This article investigates when homotopies can be converted to monotone homotopies without increasing the lengths of curves.
A monotone homotopy is one which consists of curves which are simple or constant, and in which curves are pairwise disjoint.
We show that, if the boundary of a Riemannian disc can be contracted through curves of length less than $L$, then
it can also be contracted monotonically through curves of length less than $L$.  This proves a conjecture of Chambers and Rotman.  Additionally, any sweepout of a Riemannian $2$-sphere through curves of length less than $L$ can be replaced with a monotone sweepout through curves of length less than $L$.  Applications of these results are also discussed.
\end{abstract}
\maketitle


\section{Introduction}

The primary objects of study in this article are \emph{monotone homotopies}, which we define below. Throughout the article, we consider closed curves on Riemannian surfaces. If $\alpha$ is a simple closed contractible curve, then $D(\alpha)$ denotes the closed disc that $\alpha$ bounds.
If the underlying surface has at least one boundary component, then this disc is unique.  If it is an oriented sphere, then the orientation of the sphere and the orientation of $\alpha$ determines $D(\alpha)$; it is the unique disc for which, given the orientation of the sphere, the induced orientation of the boundary agrees with that of $\alpha$. If $\alpha$ and $\beta$ are two simple closed contractible curves with $D(\beta) \subset D(\alpha)$, then let $A(\alpha, \beta) = A(\beta, \alpha)$ denote the annulus between $\alpha$ and $\beta$, that is, $D(\alpha)$ with the interior of $D(\beta)$ removed.  If $\beta$ is a constant curve, then we extend the definition of $A(\alpha, \beta)$ to denote $D(\alpha)$.

\begin{Def}
	Let $(M,g)$ be a Riemannian annulus with boundaries $\gamma_0$ and $\gamma_1$, and let $H: \mathbb{S}^1 \times [0,1] \rightarrow M$ be a homotopy between $\gamma_0$ and $\gamma_1$, that is, a smooth map such that $H(t,0)=\gamma_0$ and $H(t,1)=\gamma_1$. We will say that $H$ is \emph{monotone} if every intermediate curve $\gamma_{\tau} := H(t,\tau)$ is a simple closed curve parameterized by $t$, and if the closed $2$-annuli $A(\gamma_\tau,\gamma_1) \subseteq M$ satisfy the inclusion $A(\gamma_{\tau_2},\gamma_1) \subset A(\gamma_{\tau_1},\gamma_1)$ for every $\tau_1 < \tau_2$.  In this definition, $\gamma_0$ and $\gamma_1$ can be constant curves or simple closed curves.

	A \emph{monotone contraction} of a Riemannian 2-disc is a monotone homotopy from its boundary to a constant curve.  We say that such a monotone homotopy is \emph{outward} if $D(\gamma_0) \subset D(\gamma_1)$, and is called
	\emph{inward} if $D(\gamma_1) \subset D(\gamma_0)$.
\end{Def}

We prove the following two theorems, the first of which was a conjecture by Chambers and Rotman~\cite[Conjecture 0.2]{cr-mhcdrs-16}.
\begin{Thm}
	\label{thm:disc}
	Suppose that $(D,g)$ is a Riemannian disc, and suppose that there is a contraction of $\partial D$ through curves of length less than $L$.
	Then there is a monotone contraction of $\partial D$ through curves of length less than $L$.
\end{Thm}

The techniques involved in the proof of this theorem also apply\footnote{The proof is even simpler in that case, since case $b$ of Proposition~\ref*{Prop:zigzag_area_extend} never occurs.} in the setting of a Riemannian annulus and a homotopy between its two boundaries through curves of length less than $L$, yielding a monotone homotopy through curves of length less than $L$.

The second theorem concerns a similar monotonicity result for sweepouts of $2$-spheres.  A sweepout of a Riemannian $2$-sphere is a map
$f: S^1 \times S^1 \rightarrow S^2$ of degree $1$. We can regard a sweepout as a $1$-parameter family of connected closed curves $f(t,\cdot)$
parametrized by $t \in S^1$.  These curves might have self-intersections as well as pairwise intersections.

\begin{Thm}
	\label{thm:sphere}
	Suppose that $(S^2, g)$ is a Riemannian $2$-sphere, and suppose that $f$ is a sweepout of it composed of curves of length less than $L$.
	Then there exists a diffeomorphism from the round sphere $(S^2, round) = \{ (x,y,z) : x^2 + y^2 + z^2 = 1 \}$ to $(S^2, g)$ such that the length of the image of each parallel $\{ (x,y,z) : z = constant \} \cap (S^2, round)$ is less than $L$.
\end{Thm}

The proof of this result holds also if we assume only that there exists such a map of odd degree (which is not necessarily equal to $1$).

\subsection*{Background and related work}

These theorems have numerous applications to metric geometry, and to applied topology.  In terms of metric geometry, Theorem~\ref*{thm:disc} improves known estimates of the lengths of the shortest geodesics between pairs of points on Riemannian $2$-spheres from \cite{NR1} and \cite{NR2}.  In particular, the two authors prove that there are at least $k$ geodesics joining any two points on a Riemannian $2$-sphere of length at most $22kd$, where $d$ is the diameter of the $2$-sphere (if the two points agree, then this improves to $20kd$).  These results improve these bounds to $16kd$ and $14kd$ respectfully, and also greatly decrease the complexity of the proofs in \cite{NR1} and \cite{NR2}.

These results also allow the results from \cite{NR3} to be generalized to the free loop space of a Riemannian $2$-sphere; a map from $S^m \rightarrow \Lambda M$, where $M$ is a Riemannian $2$-sphere and $\Lambda M$ is the set of closed curves in $M$, can be homotoped to a map $\tilde{f}: S^m \rightarrow \Lambda M$  consisting of of curves of lengths bounded by $\rho(m,k,d)$, with $\rho$ being an explicit function, $d$ being the diameter of $M$, and $k$ being the number of distinct non-trivial periodic geodesics of $M$ with length at most $2d$.  In \cite{NR3}, Nabutovsky and Rotman proved the analogous statement for maps into the space of simple closed curves based at a fixed point; our results allow this restriction to be removed.  For more details, we refer to Chambers and Rotman~\cite[Section~0.1]{cr-mhcdrs-16}.  Of special note is that Theorem \ref*{thm:disc} directly  implies that, if the boundary of a Riemannian disc is contractible through curves of length less than $L$, then for any point $q$ on the boundary of the disc, the boundary is contractible to the point $q$ through loops based at $q$ of length less than $L + 2d$.  Here, $d$ is the diameter of the Riemannian $2$-disc.

The sweepouts described in Theorem \ref*{thm:sphere} appear in minimal surface and min-max literature.  In \cite{CLoptimalsphere}, Chambers and Liokumovich show that if there is a sweepout of a Riemannian $2$-sphere through curves of length less than $L$, then there is a sweepout of the same Riemannian $2$-sphere through simple closed curves and constant curves of lengths less than $L$.  They then use this result to answer a question of Freedman about the minmax levels with respect to different classes of sweepouts.  Our theorem is an improvement on this result, proving that such a sweepout can be simplified to not only consist of curves which do not have self-intersections (other than constant curves), but to consist of such curves which are (mostly) pairwise disjoint as well.

From the computational topology literature, much recent work has focused on computing a ``best" homotopy between two curves as a means of measuring similarity of the curves or determining optimal morphs between them~\cite{CVE08,cw2013,hnss-hwdmml-16}.  The main goal in this setting is to determine the computational complexity of such a problem in the most common settings, generally where the two curves are in the plane (possibly with obstacles) or on a meshed surface, as typically produced by surface reconstruction algorithms.

The type of optimality we study in this work has been investigated in a combinatorial setting, where it was called the ``height" of the homotopy~\cite{esa-hh2017,homotopyheight,hnss-hwdmml-16}, and in the graph theoretic setting, where it was called a ``b-northward migration"~\cite{Brightwell09submodularpercolation}.   However, the exact complexity of this problem remains open, and both papers include a conjecture that the best such morphings will proceed monotonically.  The monotonicity result we present in this paper is a key ingredient in showing that this problem lies in the complexity class $\mathcal{NP}$~\cite{cmo-coh-17}.

\realfig{figurecounter}{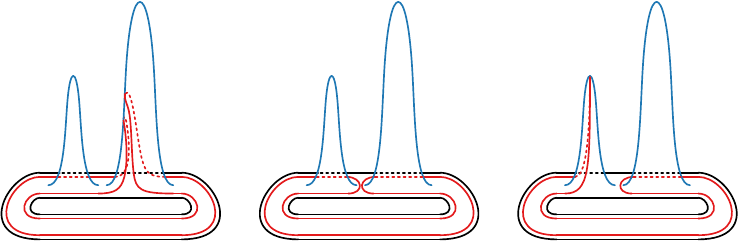}{A counter-example to Conjecture 0.3 of Chambers and Rotman~\cite{cr-mhcdrs-16}}{0.6\hsize}

Finally, Chambers and Rotman formulated another conjecture~\cite[Conjecture 0.3]{cr-mhcdrs-16} on monotonicity, where the initial curve is not the boundary of the disc. We say that a monotone contraction covers a simple closed curve $\gamma$ if $\gamma$ is contained in the disc which is the image of that contraction. They conjectured that if $M$ is a Riemannian surface and $\gamma$ a simple closed curve contractible through curves of length less than $L$, then there is a monotone contraction covering $\gamma$ through curves of length less than $L$. We observe that this conjecture is false by exhibiting the counter-example in Figure~\ref*{figurecounter}.

In that example, the underlying surface is an annulus, and the metric is the Euclidean one, except for two mountains, one taller than the other one. The initial closed curve $\gamma$ lies as shown in the first picture, half-way up the tall mountain from both sides. An optimal contraction is pictured in the following two pictures; it first climbs down the tall mountain \emph{for both sides of the curve} in order to reduce the length, before climbing over the smaller one. On the other hand, any monotone contraction covering $\gamma$ must start at a closed curve $\alpha$ that also lies half-way up the tall mountain from both sides. Then, by monotonicity, only one side of the curve can climb down the tall mountain. Therefore, the maximum length of the curves in such a monotone contraction will need to be strictly larger than for a non-monotone one.

\section{Preliminaries}

We begin by recalling several definitions.  A \emph{Riemannian disc} is a closed smooth $2$-dimensional Riemannian manifold with boundary that is diffeomorphic to a closed unit disc $D = \{ (x,y) \in \mathbb{R}^2: x^2 + y^2 \leq 1 \}$ with a smooth Riemannian metric; throughout this article, we will call this simply a ``disc".  A \emph{Riemannian annulus} is a $2$-dimensional smooth
Riemannian manifold with boundary that is diffeomorphic to an annulus $\{(x,y) \in \mathbb{R}^2 : 1 \leq x^2 + y^2 \leq 4 \}$ endowed with a smooth Riemannian metric; throughout this
article, we will refer to such a manifold with boundary simply as an ``annulus."  A closed curve in a smooth manifold with boundary $M$ is a smooth map from $S^1$ to  $M$; a simple closed
curve is a closed curve which is injective.  An arc in a smooth manifold with boundary $M$ is a smooth map from $[0,1]$ to $M$; an arc of a closed curve in $M$ is simply the restriction of the smooth map from $S^1$ to $M$ to a closed subinterval of $S^1$.  A homotopy of curves in a smooth manifold with boundary $M$ is a smooth map $H: [0,1] \times S^1 \rightarrow M$.  We remark that, at several points in this article, we will form a new curve taking a curve and replacing an arc of that curve with a new arc.  This will create at most two points which are not smooth, however, the resulting curve can be replaced by a smooth curve with length increased an arbitrarily small amount, and such that the new curve agrees with the old curve outside of
balls of arbitrarily small radii centered at the two singular points.  If the original curve is simple, then the new, smooth, curve is also simple.  In this article, we implicitly assume that
this smoothing procedure is executed whenever we execute such a cut-and-paste operation; we don't explicitly mention it to simplify the exposition.

Throughout the article, a closed curve $\gamma$ in a Riemannian annulus $A$ is called a \textit{minimizing geodesic} if it is \textit{essential} (homotopic to one of the boundaries), and its length is minimal among the essential curves.

\begin{Def}
	\label{Def:zigzag}
	A \emph{zigzag} $Z$ is a collection of homotopies $H_1, \dots, H_n$ with the following
	properties:
	\begin{enumerate}
		\item	$H_i(1) = H_{i+1}(0)$
		\item	$H_i$ alternates between outward and inward monotone homotopies, i.e., each of the $H_i$ is a monotone homotopy, but for any $i\in \{1, \ldots, n-1\}$, the concatenation of $H_i$ and $H_{i+1}$ is not.
	\end{enumerate}
	We define $\gamma_0 = H_1(0)$ and $\gamma_i = H_i(1)$ for $1 \leq i \leq n$.
	
	Each $H_i$ goes from $ \gamma_{i-1}$ to $\gamma_i$.
	We define the \emph{order} of $Z$, $ord(Z)$, to be $n$.


\end{Def}

\realfig{figure0}{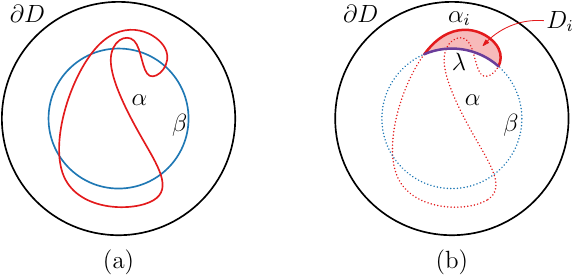}{Meandering curves}{0.7\hsize}

We will also need the following definitions and a theorem from the article of Chambers and Rotman~\cite{cr-mhcdrs-16}. 

\begin{Def} (\cite[Definition~0.6]{cr-mhcdrs-16})
	\label{Def:simplint}
	Let $\alpha:[0,1] \longrightarrow M$ and $\beta:[0,1] \longrightarrow M$ be two simple closed curves in a Riemannian manifold $M$. 
	If every two points of intersection between $\alpha$ and $\beta$ are consecutive on $\alpha$ if and only if they are consecutive on $\beta$, then $\alpha$ and $\beta$ are said to satisfy the simple intersection property. 
\end{Def}

When $\alpha$, $\beta$ defined in ~\ref*{Def:simplint} do not satisfy the simple intersection property, we will say that they are {\it meandering} with respect to each other.


\begin{Def}
\label{Def:corarc}
Let $\alpha$ and $\beta$ be two simple closed curves in a closed topological 
$2$-disc $D$. Let $\alpha_i=\alpha|_{[t_i,t_{i+1}]}$ be an arc of $\alpha$, such that
the interior of the arc does not intersect $\beta$, while its 
endpoints $\alpha(t_i), \alpha(t_{i+1}) \in \beta$. Then these points subdivide $\beta$ into 
two arcs. Let $\lambda$ be an arc that together with $\alpha_i$ bounds a disc in the closed annulus
$A(\partial D, \beta(t))$ between $\partial D$ and $\beta$. Then we will call $\lambda$ a {\it corresponding arc}. We will refer to the disc $D_i$ with the boundary $\alpha_i\cup \lambda$ as a {\it corresponding disc}. (See fig. ~\ref*{figure0} (b), where the disc that corresponds to arc $\alpha_i$ is shaded.)  Note that $\alpha$ and $\beta$ may intersect an infinite number
of times; this definition still holds.
\end{Def}


We first prove that any (non-monotone) homotopy can be approximated by a zigzag.

\begin{Prop}
	\label{Prop:zigzag_reduction}
	Suppose that there is a contraction of $\partial D$ through
	curves of length less than $L$. Then there exists a zigzag of order $n$ such that $\gamma_0 =
	\partial D$ and $\gamma_n$ is a constant curve, and such that
	all curves of all homotopies have length less than $L$.
\end{Prop}

	\begin{proof}
First, by a result of Chambers and
Liokumovich~\cite[Theorem~1.1]{cl-chidhh-14}, we know that there exists a contraction of $\partial D$ through
\textit{simple} closed curves of length less than $L$.

We say that a corresponding disc between two arcs $\alpha$ and $\alpha'$ is $\delta$-thin if the homotopic Fr\'echet distance between the two curves is less than $\delta$, where the homotopic Fr\'echet distance between two curves is defined by considering all homotopies $H: [0,1] \times [0,1] \rightarrow D$ from $H(\cdot,0) = \alpha$ to $H(\cdot,1) = \alpha'$ (up to reparametrizations), and taking  $\inf_{H} \sup_{s \in [0,1]} \text{length}(H(s,\cdot))$.  In other words, for each homotopy from $\alpha$ to $\alpha'$ (allowing reparametrizations), we consider the length of the longest curve $H(s, \cdot)$ in that homotopy; taking the infimum of
this quantity over all of these homotopies yields the homotopic Fr\'echet distance~\cite{CVE08}.
Similarly, an annulus $A(\alpha, \beta)$ is $\delta$-thin if the two boundary curves have homotopic Fr\'echet distance less than $\delta$. Now, we consider a discretized contraction $H$; we consider an increasing sequence of $n$ values $t_1, \dots, t_n$ in $[0,1]$ so that

\begin{itemize}
	\item $t_0=0$ and $H(0) = \partial D$, 
	\item $H$ on $[t_i, t_{i+1}]$ is a homotopy through curves of length less than $L$,
	\item $t_n=1$ and $H(1)$ is a constant curve, and
	\item for $0 \leq i \leq n-1$, if $H(t_i)$ and $H(t_{i+1})$ intersect, they have the simple intersection property and the corresponding discs are $\delta$-thin, for $\delta$ to be determined later. If they do not intersect, the annulus $A(H(t_i),H(t_{i+1}))$ is $\delta$-thin.
\end{itemize}

We begin with the original contraction $\tilde{H}$ produced by the theorem of Chambers and Liokumovich. Without loss of generality, we may assume that the only constant curve in the contraction occurs at $t = 1$.  If this is not the case, then we choose $t_0$ to be the smallest element of $[0,1]$ such that $\tilde{H}(t_0)$ is
a constant curve, and we apply the rest of the argument to the contraction formed by restricting $\tilde{H}$ to $[0,t_0]$.

Next, we select $t^* \in (0,1)$ sufficiently large so that $\tilde{H}(t^*)$
can be contracted in $D(\tilde{H}(t^*))$ through disjoint closed curves of length less than $L$, which are all simple except for the final constant curve.
Let this contraction of $\tilde{H}(t^*)$ be denoted by $K: [t^*,1] \times S^1 \rightarrow D$.

Since $H$ is smooth and the interval $[0,t^*]$ is compact, there is a $\delta > 0$ and an $\epsilon > 0$ such that, for every $s_1, s_2 \in [0,t^*]$, if $|s_1 - s_2| < \epsilon$, then $\tilde{H}(s_2)$ is contained in the $\delta$-tubular
neighborhood of $\tilde{H}(s_1)$.  Furthermore, there are parametrizations $\gamma_1$ and $\gamma_2$ 
of $\tilde{H}(s_1)$ and $\tilde{H}(s_2)$, respectfully, such that $\gamma_2(x) = \gamma_1(x) + f(x) v(x)$, where $v(x)$ is the outward unit vector to $\gamma_1$, and $f(x)$ is a real number with $|f(x)| < \delta$ (note that outward is with respect to $D(\gamma_1)$).  Due to this property, if $\tilde{H}(s_1)$ and $\tilde{H}(s_2)$ do not intersect, then $A(\tilde{H}(s_1), \tilde{H}(s_2))$ is $\delta$-thin, and if they do intersect, then they have the simple intersection property, and the corresponding discs are all $\delta$-thin.

To complete the proof, let $n = \lceil{\frac{t^*}{\epsilon}}\rceil$, and take our discretized sequence to be $t_i = i \frac{t^*}{n}$ for $i \in \{ 0, \dots, n\}$, and the contraction $H$ is defined
as $H = \tilde{H}$ on $[t_i, t_{i+1}]$ for all $i \in \{ 0, \dots, n-1 \}$.  Since $K$ is monotone, we can find a sufficiently large positive integer $m$ so that setting $t_i = t^* + (i - n) \frac{1 - t^*}{m}$ for $i \in \{ n, \dots, n + m \}$, and setting $H = K$ on $[t_i, t_{i+1}]$ for $i \in \{ n, \dots, n + m - 1 \}$ completes the proof.

Now, if $H(t_i)$ and $H(t_{i-1})$ intersect, for each $0<i<n$, we define an auxiliary curve $H(t_i)^f$ from  $H(t_i)$: $H(t_i)^f$ is obtained from $H(t_i)$ by considering all of the arcs of $H(t_i)$ in $D(H(t_{i-1}))$ and replacing the other ones by the arcs they correspond to in $H(t_{i-1})$. Then we claim that there are monotone homotopies between $H(t_i)$ and $H(t_i)^f$, and between $H(t_i)^f$ and $H(t_{i+1})$ such that the intermediate curves have length less than $L$. Indeed, one can go from one to the other using monotone homotopies that interpolate within the corresponding discs, and if $\delta$ is chosen small enough, this interpolation can be done with an arbitrarily small overhead on the lengths of the curves. If $H(t_{i})$ and $H(t_{i-1})$ do not intersect, for $\delta$ small enough, the $\delta$-thin assumption implies that there exists a monotone homotopy between $H(t_{i-1})$ and $H(t_i)$, such that 
 the intermediate closed curves have length less than $L$.

Gluing together all of these monotone homotopies, we obtain a zigzag with curves of length at most $L$.
\end{proof}

One of our main technical tools is a technique to modify a monotone homotopy when it crosses a minimizing geodesic. The general setting is when we have a monotone homotopy $H$ between two curves $\alpha_0$ and $\alpha_1$, and $\alpha$ is a third simple closed curve that is a minimizing geodesic in $A(\alpha,\alpha_0)$. Then we can use $\alpha$ to ``shortcut'' the homotopy $H$. Depending on the relative positions of $\alpha_1$, there are three variants of this shortcutting argument, leading to three different outcomes: they are summarized in the following proposition.

\realfig{figure3}{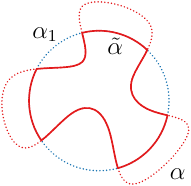}{Construction of $\tilde{\alpha}$.}{0.7\hsize}

\begin{Prop}\label{Prop:homotopyglue}
  Let $H$ be a monotone homotopy between simple closed curves $\alpha_0$ and $\alpha_1$ such that the intermediate curves have length less than $L$ and let $\alpha$ be another simple closed curve, disjoint from $\alpha_0$ and such that $\alpha$ is a minimizing geodesic in $A(\alpha_0,\alpha)$. Then:
  \begin{enumerate}
  \item If $\alpha$ is entirely contained and essential in $A(\alpha_0,\alpha_1)$, then there exists a monotone homotopy between $\alpha$ and $\alpha_1$ where the intermediate curves have length less than $L$.
  \item If $\alpha$ is entirely contained and non-essential in $A(\alpha_0,\alpha_1)$, then there exists a monotone homotopy between $\alpha$ and a constant curve $p$ where $p$ is a point on $\alpha$, and where the intermediate curves have length less than $L$.
  \item If $\alpha$ has the simple intersection property with $\alpha_1$, then if we denote by $\tilde{\alpha}$ the curve obtained from $\alpha_1$ by replacing segments of $\alpha_1$ in $A(\alpha_0,\alpha)$ with the corresponding arcs (see Figure~\ref*{figure3}), there exists a monotone homotopy between $\alpha$ and $\tilde{\alpha}$ where the intermediate curves have length less than $L$.
  \end{enumerate}
\end{Prop}

Although not explicitly stated in Chambers and Rotman~\cite{cr-mhcdrs-16}, the first case of this Proposition is implicit in the proof of their Theorem 0.7. More precisely, their proof is divided in two steps, and this is the result obtained in Step 1. The proof of the other two cases of Proposition~\ref*{Prop:homotopyglue} also follows closely the arguments of the proof of Theorem~0.7. For the sake of completeness, we include the full proof below.

\begin{proof}
  The general idea is that in all three settings one can take each intermediate curve of the homotopy $H$ and replace the portions that are outside of the target annulus using $\alpha$. More precisely, if we denote by $(\alpha_t)_{t\in [0,1]}$ the curves of the homotopy $H$, we do the following in each case:

  \begin{enumerate}
\item Let $t_0$ denote the first time $t$ that $\alpha_t$ intersects $\alpha$. Then for all $t \geq t_0$, we replace in each curve $\alpha_t$ the segments that are in the interior of the annulus $A(\alpha_0,\alpha)$ with their corresponding arcs and consider the family $A$ of the closed curves $(\alpha_t)_{t\in [t_0,1]}$.
\item Let $t_0$ and $t_1$ denote respectively the first and last time $t$ that $\alpha_t$ intersects $\alpha$. Then for all $t_0\leq t \leq t_1$, we replace in each curve $\alpha_t$ the segments that are in the interior of the annulus $A(\alpha_0,\alpha)$ with their corresponding arcs. Note that $\alpha_{t_1}$ is a contractible curve of $\alpha$ (viewed as a set), and can thus be contracted to a point $p \in \alpha$ while monotonically decreasing its length. This contraction is realized through curves $(\alpha_t)_{t\in [t_1,1]}$. We now consider the family $A$ of closed curves $(\alpha_t)_{t\in [t_0,1]}$.
\item Let $t_0$ denote the first time $t$ that $\alpha_t$ intersects $\alpha$. Then for all $t \geq t_0$, we replace in each curve $\alpha_t$ the segments that are in the interior of the annulus $A(\alpha_0,\alpha)$ with their corresponding arcs and consider the family $A$ of the closed curves $(\alpha_t)_{t\in [t_0,1]}$. Note that the new $\alpha_1$ coincides with $\tilde{\alpha}$.
  \end{enumerate}

 Now, the rest of the proof is the same in all three cases. Since $\alpha$ is minimizing in $A(\alpha_0,\alpha)$, the corresponding arcs are always shorter than the arcs that they replace. Thus, the families $A$ that we obtain contain intermediate curves of length less than $L$. Furthermore, $\alpha_{t_0}$ and $\alpha_1$ are the starting and ending curves of the target homotopy in all three cases. However, the families $A$ fail to be monotone homotopies because they are neither homotopies (there can be discontinuities) nor monotone (the curves are not even simple). The first issue is solved by interpolation and the second one by perturbation.

  \realfig{figure4}{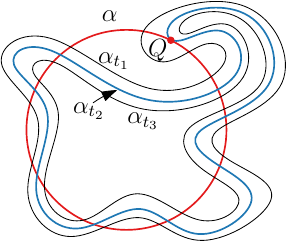}{Discontinuities might occur...}{0.7\hsize}

\realfig{figure5}{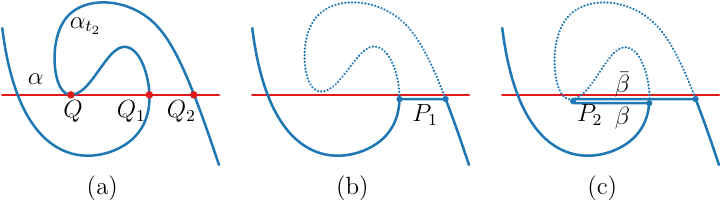}{... but they can be fixed by interpolating with a path homotopy.}{0.6\hsize}

Discontinuities only appear at times $t$ when the intersection between
$\alpha_t$ and $\alpha$ is not transversal. Figure~\ref*{figure4}
depicts such a situation.  Here $\alpha_{t_2}$ touches $\alpha$ at
point $Q$. There are two ways to replace the segments of
$\alpha_{t_2}$ in the neighborhood of $Q$, (see Figure~\ref*{figure5}
(a) that depicts this situation locally).  One way is to replace the
segment of $\alpha_{t_2}$ that connects the points $Q_1$ and $Q_2$
that lies in the annulus $A(\alpha_0,\alpha)$  by the path $P_1$,
(see Figure~\ref*{figure5} (b)). Let us call this replacement the type
1 replacement.  Another way is depicted in Figure~\ref*{figure5}
(c). Here we replace the segment of $\alpha_{t_2}$ that connects $Q_1$
and $Q_2$ by $P_2$.  $P_2$ is a path that consists of two paths: the
first one replaces the segment of $\alpha_{t_2}$ that connects $Q$ and
$Q_2$, while the second one, $\beta$, replaces the segment of
$\alpha_{t_2}$ that connects $Q_1$ and $Q$. Let us call this
replacement the type 2 replacement. Since our procedure only replaces
segments in the \emph{interior} of the annulus $A(\alpha_0,\alpha)$
with their corresponding arcs, it always chooses the type 2
replacement.

Now, there is only one way to replace the relevant part of $\alpha_{t_1}$, with a curve that
is close to $\alpha_{t_2}$ and is contained in $A(\alpha_0,\alpha_{t_2})$ (Figure~\ref*{figure4}).  If we want the procedure to result
in a homotopy, this fits well with our choice of type 2 replacement on
$\alpha_{t_2}$. On the other hand, there is also only one type of replacement
that can be performed on $\alpha_{t_3}$,  with a curve that is close to
$\alpha_{t_2}$ and is contained in $A(\alpha_{t_2},\alpha_1)$. And as $t_3$ goes to $t_2$, this converges into a type 1 replacement for
$\alpha_{t_2}$. Hence, we have a discontinuity at $t_2$. To avoid
this discontinuity, note that $P_2=\beta *\bar{\beta}*P_1$ (see
Figure~\ref*{figure5} (c)). Here, $\bar{\beta}$ denotes the path $\beta$
traversed in the opposite direction, and $a * b$ denotes the curve formed by concatenating the
curves $a$ and $b$. Therefore, $P_1$ and $P_2$ can be
connected by the obvious length non-increasing path homotopy, which
amounts to contracting $\beta*\bar{\beta}$ to $Q_1$. This path
homotopy extends to a homotopy between the two curves derived from the two replacement choices for
$\alpha_{t_2}$. Thus, including the homotopy between the two
different resulting curves corresponding to type 1 and type 2 replacements solves the discontinuity problem at time $t_2$. Doing so for each time $t$ when the intersection between $\alpha_t$ and $\alpha$ is not transverse makes $A$ into a homotopy between $\alpha_{t_0}$ and $\alpha_1$.

Finally, observe that while the curves in the homotopy $A$ may not be simple, they do not feature transversal intersections, since the shortcutting procedure replaces all the segments outside of the annulus $A(\alpha_0,\alpha)$ by their corresponding arcs. Furthermore, since $H$ was a monotone homotopy, there was no transverse intersection between $\alpha_t$ and $\alpha_{t'}$  for $t\neq t'$ before the replacement procedure, and thus by the same argument there are none afterwards either. Now, by applying an arbitrarily slight perturbation to all of the curves in $A$ in a continuous way, we can make all the curves simple while still having no transverse intersection pairwise and all having length less than $L$. This yields a monotone homotopy and concludes the proof.
\end{proof}

The proof of Theorem~\ref*{thm:disc} relies on Propositions~\ref*{Prop:zigzag_length_shorten} and~\ref*{Prop:zigzag_area_extend}, which allow us to modify small portions of zigzags. The first one follows rather directly from Proposition~\ref*{Prop:homotopyglue}, but the second one requires more work. The proof of Theorem~\ref*{thm:sphere} relies on Proposition~\ref*{Prop:zigzag_area_extend} and a small variant of Proposition~\ref*{Prop:zigzag_length_shorten}, which is stated in Proposition~\ref*{Prop:sphere}.

\begin{Prop}
	\label{Prop:zigzag_length_shorten}
	Suppose that $Z$ is an order 2 zigzag through curves of length less than $L$.  If $\gamma_1$ is not a minimizing geodesic in 
	$A(\gamma_1, \gamma_2)$, and if a minimizing geodesic $\gamma$ in this annulus also lies
	in the interior of $A(\gamma_0, \gamma_1)$ and is essential in it, then there is a zigzag $Z'$ where the intermediate curves have length less than $L$ and such that
	\begin{enumerate}
		\item	$ord(Z') = 2$.
		\item	$\gamma_1'$ minimizes in $A(\gamma_1', \gamma_2')$.
		\item	$\gamma'_0 = \gamma_0$, $\gamma_1' = \gamma$, and $\gamma'_2 = \gamma_2$.
	\end{enumerate}

	Suppose that $Z$ is an order $2$ zigzag where the intermediate curves have length less than $L$, and that $\gamma_0$ is a minimizing geodesic
	in $A(\gamma_1, \gamma_2)$, or that $\gamma_2$ is a minimizing geodesic in $A(\gamma_0, \gamma_1)$.
	Then there is an order $1$ zigzag $Z'$ (i.e., a monotone homotopy) through curves of length less than $L$ and such that $\gamma'_0 = \gamma_0$, and $\gamma'_1 = \gamma_2$.
\end{Prop}

\begin{proof}
The first part of the proposition follows from two applications of Proposition~\ref*{Prop:homotopyglue}(1). We first apply it to the reversal of the homotopy $H_1$ and the curve $\gamma$, and then to the homotopy $H_2$ and the curve $\gamma$. This yields two new homotopies $H'_0$ and $H'_1$, going respectively from $\gamma_0$ to $\gamma$ and from $\gamma$ to $\gamma_2$; their concatenation satisfies the needed properties.

For the second part of the proposition, let us first deal with the first case where $\gamma_0$ is a minimizing geodesic in $A(\gamma_1,\gamma_2)$. Then one application of Proposition~\ref*{Prop:homotopyglue}(1) to the homotopy $H_2$ and $\gamma_0$ yields the homotopy from  $\gamma_0$ to $\gamma_2$. The other case is obtained by applying the theorem to the reversal of $H_1$ and $\gamma_2$ instead.
\end{proof}

\begin{Lem}
	\label{Lem:minimizer_simple}
	Suppose that $Z$ is a zigzag of order $2$, and that $\gamma_0$ is a minimizer in $A(\gamma_0, \gamma_1)$.
	Then there exists an essential curve $\gamma$ which is a minimizer in $A(\gamma_1, \gamma_2)$, and which
	has the simple intersection property with $\gamma_0$.
\end{Lem}

\begin{proof}[Proof]
	We begin by choosing an essential minimizing curve $\alpha$ in $A(\gamma_1,\gamma_2)$.
	Let $\rho$ be a segment of $\gamma_0$ whose endpoints are intersections between $\alpha$ and
	$\gamma_0$, and whose interior is contained in the interior of $A(\gamma_1, \alpha)$.  Let
	the endpoints of $\rho$ be $\rho(0)$ and $\rho(1)$.

	From $\alpha$ and $\rho$ we can define two auxiliary curves: one which goes from $\rho(0)$ to $\rho(1)$ following $\alpha$ and then
	back to $\rho(0)$ along $\rho$, and one which goes from $\rho(1)$ to $\rho(0)$ following $\alpha$ and then back to $\rho(1)$
	along $\rho$.  Let the first curve be $\beta_1$, and let the other one be $\beta_2$.  Note that both
	$\beta_1$ and $\beta_2$ are contained in $A(\gamma_1, \alpha)$ and are simple closed curves. Thus, in this annulus, exactly one of $\beta_1$ or $\beta_2$ is essential; without loss of generality, we may assume
	that it is $\beta_1$. Since $\beta_2$ is not essential, it bounds a disc within $A(\gamma_1, \alpha)$ which we call by a slight abuse of language its \emph{interior}. If $\rho$ lies in the boundary of
	a portion of the interior of $\beta_2$ outside of $A(\gamma_0, \gamma_1)$, as in the middle picture of Figure~\ref*{figurerho} then we do nothing.

	If $\rho$ lies on the boundary of a portion of the interior of $\beta_2$ inside of 
	$A(\gamma_0, \gamma_1)$, as in the right picture of Figure~\ref*{figurerho}, we claim that $\beta_1$ has total length not greater than that of $\alpha$. Indeed, let us first build an auxiliary closed curve in the following way: take all the segments $\Sigma$ of $\beta_2$ that lie in the interior of $A(\gamma_0,\gamma_1)$ and have their endpoints on $\gamma_0$. The segments in $\Sigma$ are also segments of $\alpha$, and since $\gamma_0$ is a minimizing geodesic in $A(\gamma_0,\gamma_1)$, any $\sigma \in \Sigma$ is at least as long as its corresponding arc on $\gamma_0$. The closed curve $\alpha'$ is obtained by replacing all the segments in $\Sigma$ from $\alpha$ by their corresponding arcs on $\gamma_0$; it is not longer than $\alpha$. Now, $\alpha'$ may not be simple, in particular there may be double points on $\rho(0)$ or $\rho(1)$. Such a double point $\alpha'(t_1)=\alpha'(t_2)$ for $t_1 \neq t_2$ cuts $\alpha'$ into two subcurves, one of which, say $\alpha'_{|[t_1,t_2]}$, is contractible in $A(\gamma_1,\gamma_2)$. One can shortcut $\alpha'$ even more by removing such a contractible portion, i.e., replacing $\alpha'$ by the closed subcurve $\alpha'_{|[t_2,t_1]}$. After removing all these contractible subcurves, we obtain the curve $\beta_1$, which is by construction not longer than $\alpha'$ and thus not longer than $\alpha$.

	\realfig{figurerho}{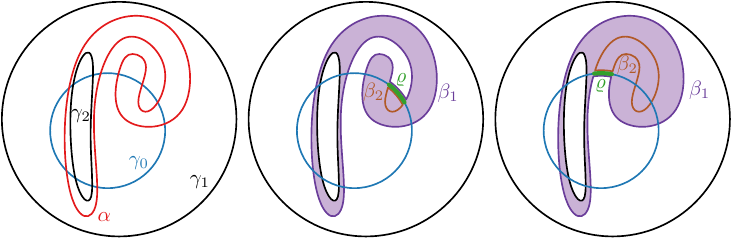}{The different curves in the proof of Lemma~\ref*{Lem:minimizer_simple}. For the $\rho$ in the middle diagram, we do nothing, while for the $\rho$ in the right diagram we can shortcut $\alpha$ by replacing it with $\beta_1$.}{0.6\hsize}

	We now build $\gamma$ in the following way.  Since $\gamma_0$ has bounded length, there are countably many segments of $\gamma_0$ which satisfy the above properties.  Let these segments be $\rho_1, \rho_2, \dots$; we apply the above procedure to $\alpha = \omega_0$ and $\rho_1$ to obtain a curve $\omega_1$.
If a segment of $\gamma_0$ satisfies the above properties with respect to $\omega_1$, then it also satisfies those properties with respect to
$\omega_0$ ($=\alpha$).  If $\rho_2$ is still one of these segments, then we apply the procedure to $\omega_1$ and $\rho_2$ to form $\omega_2$.  We continue to do this for all
$\rho_i$ to form a sequence of curves $\omega_1, \omega_2, \dots$.  All of these curves are minimizers in $A(\gamma_1, \gamma_2)$, and so all lie in this
annulus, all are smooth, and all have length bounded by $L$.  By the Arzel\`{a}-Ascoli theorem, there is a curve to which these curves converge; we let this curve be
$\gamma$.  $\gamma$ is still a minimizer in $A(\gamma_1, \gamma_2)$, and $\gamma$ has the simple intersection property with $\gamma_0$.  If it did
not have the simple intersection property with $\gamma_0$, then there is a segment $\rho$ of $\gamma_0$ whose endpoints are also in $\gamma$, whose interior is contained
in the interior of $A(\gamma_1, \gamma)$, and which produces a $\beta_1$ and a $\beta_2$ which are in the second case.  However, then
$\rho = \rho_i \in \{ \rho_1, \rho_2, \dots \}$, and so $\rho = \rho_i$ would not satisfy the above properties with respect to $\omega_i$, and so it would also
not satisfy the above properties with respect to $\gamma$, yielding a contradiction.
\end{proof}

\begin{Prop}
	\label{Prop:zigzag_area_extend}
	Suppose that $Z$ is a zigzag of order $3$ where the intermediate curves have length at most $L$ and such that $\gamma_1$ is a minimizing geodesic in $A(\gamma_1, \gamma_2)$, but is not a constant curve.

	Then one of the following two cases is true:
	
	{\bf Case a.} There is a zigzag $Z'$ of order $3$ such that
	\begin{enumerate}
		\item	$\gamma_0' = \gamma_0$, $\gamma_3' = \gamma_3$, and $\gamma_2' = \gamma_2$.
		\item	There exists a minimizing geodesic $\gamma \in A(\gamma_2', \gamma_3')$ which is fully contained in the interior of $A(\gamma_1', \gamma_2')$.
		\item	$\gamma_1'$ is a minimizing geodesic in $A(\gamma_1', \gamma_2')$.
		\item	All curves in $Z'$ have length less than $L$.
	\end{enumerate}

	{\bf Case b.} There exists a zigzag $Z'$ of order $1$ such that
	\begin{enumerate}
		\item	$\gamma_0' = \gamma_0$.
		\item	$\gamma_1'$ is a constant curve.
		\item	All curves in $Z'$ have length less than $L$.
	\end{enumerate}

\end{Prop}

\begin{proof}[Proof]

We begin by applying Lemma \ref*{Lem:minimizer_simple} to the order $2$ zigzag from $\gamma_1$ to
$\gamma_2$ to $\gamma_3$ to obtain an essential minimizing geodesic $\gamma$ in $A(\gamma_2, \gamma_3)$ which
has the simple intersection property with $\gamma_1$.

If $\gamma$ lies in $A(\gamma_1, \gamma_2)$, and is essential in this annulus, then we are done as
Case a is satisfied.

We now divide the remainder of the proof into two cases:
\begin{itemize}
  \item[(i)] If $\gamma$ is not entirely contained in $A(\gamma_1, \gamma_2)$, we will show that case a. holds.
  \item[(ii)] If $\gamma$ is contained in $A(\gamma_1, \gamma_2 )$, and is non-essential
there, we will show that case b. holds.
\end{itemize}

\noindent{\bf Case (i)}
Suppose that $\gamma$ is not entirely contained in $A(\gamma_1, \gamma_2)$.  Note that one can always modify the homotopy
$H_3$ to obtain a new monotone homotopy $H_3'$ between $\gamma_2$ and $\gamma_3$, such that the lengths of the curves in $H_3'$ is less than $L$ and $\gamma$ is one of the curves of $H_3'$.  
One achieves this by applying Proposition~\ref*{Prop:homotopyglue}(1) to the reversal of the homotopy $H_3$ and $\gamma$, and then to $H_3$ and $\gamma$ and concatenating the two resulting homotopies. Thus, without loss of generality, we assume that $\gamma = (H_3)_t$ for
some $ t \in [0,1]$. 


Since $\gamma$ is not entirely contained in $A(\gamma_1, \gamma_2)$, the new zigzag $Z'$ is obtained in the following manner. We define a new curve $\tilde{\gamma}$ by replacing the segments of $\gamma$ in $A(\gamma_1,\gamma_2)$ with the corresponding segments of $\gamma_1$. Then, using Proposition~\ref*{Prop:homotopyglue}(3) on $H_3$ and $\gamma_1$, we form an auxiliary monotone homotopy $\tilde{H}$ from $\gamma_1$ to $\tilde{\gamma}$.


 To form our new homotopy, we append the homotopy $\tilde{H}$ to the end of $H_1$,
and we append $\tilde{H}$ in reverse direction to the beginning of $H_2$.  Clearly, properties $1$ and $4$ are satisfied.

Property $2$ follows from the fact that $\gamma$ and $\gamma_1$ have the simple intersection property. 

To prove property $3$, we fix any essential curve $\gamma'$ in $A(\gamma_1', \gamma_2')$.  If we replace segments of $\gamma'$ which lie in $A(\gamma_1, \gamma_2)$ with segments of $\gamma_1$,
then replace segments of the resulting curve which lie in $A(\gamma_2', \gamma)$ with segments of $\gamma$, we obtain $\tilde{\gamma} = \gamma_1'$.  This procedure does not increase the length,
so this shows that the length of $\gamma'$ is greater than or equal to the length of $\gamma_1'$, completing the proof.

\medskip

\noindent {\bf Case (ii)} Suppose that $\gamma$ is contained in
$A(\gamma_1, \gamma_2)$, and is non-essential in it.  In this case,
the disc bounded by $\gamma$ does not intersect the disc bounded by
$\gamma_1$. Thus, by monotonicity, the homotopy $H_3$ ``sweeps''
$D(\gamma_1)$ completely. Thus we are in the situation to apply case
(2) of Proposition~\ref*{Prop:homotopyglue} to the homotopy $H_3$ and
$\gamma_1$. This yields a homotopy $\tilde{H}$ between $\gamma_1$ and
a point $p$ on $\gamma_1$. We then concatenate $\tilde{H}$ to the end
of $H_1$ to form $H'_1$. Clearly, both properties are satisfied.
\end{proof}

\begin{Prop}
	\label{Prop:sphere}
	Suppose that $Z$ is a zigzag of order $2$ on a Riemannian sphere such that all curves have length less than $L$,
	and such that $\gamma_0$ is a constant curve, but $\gamma_1$ and $\gamma_2$ are not constant curves.  Furthermore,
	assume that the orientation of the sphere is such that the discs bounded by curves close to $\gamma_0$ in the first monotone homotopy
	are close to the image of $\gamma_0$.
	Then there exists an order $2$ zigzag $Z'$ where $\gamma'_0$ is a constant curve in $D(\gamma_1)$, $\gamma_2' = \gamma_2$, and $\gamma_1'$ is a minimizing geodesic
	in $A(\gamma_1', \gamma_2')$.
\end{Prop}
\begin{proof}
	Let $\gamma$ be a minimizing geodesic in $A(\gamma_1, \gamma_2)$.  There are two possibilities:
	\begin{enumerate}
		\item	$\gamma_0$ is contained in both $D(\gamma_1)$ and $D(\gamma)$.
		\item	$\gamma_0$ is contained in exactly one of $D(\gamma_1)$ and $D(\gamma)$.
	\end{enumerate}
	
	If the first condition is true, then the conclusion follows from Proposition \ref*{Prop:zigzag_length_shorten}.
	If the second condition is true, then $H_1$ ``sweeps" $D(\gamma)$, and we apply case (2) of Proposition~\ref*{Prop:homotopyglue} to the reversal of $H_1$ and $\gamma$ to obtain a monotone homotopy from $\gamma$ to a point. Furthermore, applying case (1) of Proposition~\ref*{Prop:homotopyglue} to $H_2$ and $\gamma$ yields a monotone homotopy from $\gamma$ to $\gamma_2$. Concatenating these gives the result.
\end{proof}

\section{Proof of Theorems \ref*{thm:disc} and \ref*{thm:sphere}}

We first find a zigzag
which starts at the boundary of the Riemannian disc, ends at a constant curve, and traverses
curves of length less than $L$ which minimizes the order of the zigzag.




\begin{Prop}
	\label{Prop:exists}
	Suppose that there exists a contraction of $\partial D$ through curves of length
	less than $L$.  Then there is a zigzag $Z$ of finite order, which consists of curves of length less than $L$,
	and which begins on $\partial D$, and ends at a point.  Furthermore, for every zigzag $\tilde{Z}$ with these
	properties, the order of $\tilde{Z}$ is greater than or equal to the order of $Z$.
\end{Prop}

The proof follows directly from Proposition~\ref*{Prop:zigzag_reduction}, and from the fact that the order of a zigzag
is a positive integer.  We will need one more lemma before we can prove our two theorems.

\begin{Lem}
	\label{Lem:zigzag_modification}
	Suppose that $Z$ is a zigzag of order $n \geq 2$ through curves of length less than $L$, and suppose that at most
	the initial and final curves are constant.  Suppose further that $\gamma_1$ is a minimizing geodesic in $A(\gamma_1, \gamma_2)$.
	Then one of the following is true.  First, there exists a zigzag $Z'$ of order $n$ with $\gamma_0' = \gamma_0$ and $\gamma_n' = \gamma_n$, 
	every $\gamma_i$ is a minimizing geodesic in $A(\gamma_i, \gamma_{i+1})$ for all $i \in \{1, \dots, n-1\}$, and some minimizing geodesic in $A(\gamma_{i+1}, \gamma_{i+2})$ is contained and essential
	in $A(\gamma_i, \gamma_{i+1})$ for all $i \in \{ 1, \dots, n - 2 \}$.  Second, there exist two zigzags, $Z'_1$ and $Z'_2$ of orders $m_1 > 0$ and $m_2 > 0$
	with $m_1 + m_2 = n$ through curves of length less than $L$, and such that the first curve of $Z'_1$ is equal to the first curve of $Z$, the last curve of
	$Z'_2$ is equal to the last curve of $Z$, the last curve of $Z'_1$ is a constant curve, and the first curve of $Z'_2$ is equal to the same constant curve. 
\end{Lem}
\begin{proof}
	We will prove this lemma by induction on $n$, and by using
	Proposition \ref*{Prop:zigzag_area_extend} and Proposition
	\ref*{Prop:zigzag_length_shorten}. If $n=2$, there is nothing
	to prove. If $n = 3$, then we apply Proposition
	\ref*{Prop:zigzag_area_extend} to $Z$, followed by applying
	Proposition \ref*{Prop:zigzag_length_shorten} to the final
	order $2$ zigzag.  If, during the process, we produce a zigzag
	of smaller order which ends at a constant curve, then we
	terminate this procedure.

	For the inductive step, we first apply the induction
	hypothesis to the order $n-1$ zigzag at the beginning of
	$Z$. If we obtain a zigzag of smaller order which ends at a
	constant curve, then we are in the second case of the lemma
	and we are done. Otherwise, we obtain a new zigzag $Z'$ where
	$\gamma_{n-2}$ is a minimizing geodesic in
	$A(\gamma_{n-2},\gamma_{n-1})$. Thus we are in a position to
	apply Proposition \ref*{Prop:zigzag_area_extend} to the order
	$3$ zigzag at the end of $Z'$, from $\gamma_{n-3}$ to
	$\gamma_n$. If we are in case $b$ of that proposition, we are
	done since we obtain a zigzag of smaller order ending at a
	constant curve.

	Otherwise, we obtain a new zigzag, which we replace into $Z'$ to yield $Z''$. Since $\gamma_{n-2}$ may have moved in the last step, it may be the case that $\gamma_{n-3}$ is not a minimizing geodesic in $A(\gamma_{n-3},\gamma_{n-2})$ anymore. In order to fix this, we apply the induction hypothesis once again, this time to the $n-2$ zigzag at the beginning of $Z''$, yielding yet another zigzag $Z'''$ (once again, we are done if we are in the second case of the lemma). Since $\gamma_{n-2}$ and $\gamma_{n-1}$ have not been changed in this last step, we still have that $\gamma_{n-2}$ is a minimizing geodesic in $A(\gamma_{n-2},\gamma_{n-1})$, and some minimizing geodesic in $A(\gamma_{n-1}, \gamma_{n})$ is contained and essential in $A(\gamma_{n-2}, \gamma_{n-1})$. Now, either $\gamma_{n-1}$ is a minimizing geodesic in $A(\gamma_{n-1},\gamma_n)$ and we are done, or we can apply Proposition \ref*{Prop:zigzag_length_shorten} to the final order $2$ zigzag of $Z'''$. The resulting zigzag fulfills all the properties of the lemma.
\end{proof}

We now have all the tools to prove our main theorems.


\begin{proof}[Proof of Theorem~\ref*{thm:disc}]
	Let $Z$ be a zigzag which satisfies the conclusion of Proposition \ref*{Prop:exists}.  If the order
	of $Z$ is equal to $1$, then the proof is finished.  As such, assume that $ord(Z) > 1$.  Since
	the zigzag must end at a constant curve, $ord(Z) \geq 3$. We may further assume that no other curve
	in the zigzag is a constant curve.  Additionally, since $\gamma_0 = \partial D$, $A(\gamma_1, \gamma_2)$ is contained
	in $A(\gamma_0, \gamma_1)$.  As a result, we can apply Proposition \ref*{Prop:zigzag_length_shorten} to replace $Z$
	with a zigzag with the property that $\gamma_1$ is minimizing in $A(\gamma_1, \gamma_2)$ (this also uses the fact that
	we cannot produce a zigzag of shorter order from $\partial D$ to a constant curve).  As a result, we may assume
	that $Z$ has this property.

	We now apply Lemma \ref*{Lem:zigzag_modification} to $Z$.  Since we cannot find a zigzag of smaller order which begins at $\partial D$
	and ends at a constant curve, the result is an order $n$ zigzag satisfying the conclusions of the lemma.  In particular, $\gamma_{n-1}$
	must be a minimizing geodesic in $A(\gamma_{n-1}, \gamma_n)$, but must not be a constant curve.
	However, $\gamma_n$ is a constant curve, and so $\gamma_{n-1}$ must also be constant, having length $0$.  This is a contradiction, completing the proof.
\end{proof}

We can use a very similar technique to prove Theorem \ref*{thm:sphere}:

\begin{proof}[Proof of Theorem~\ref*{thm:sphere}]
	To prove Theorem \ref*{thm:sphere}, we proceed in a similar way.  We first apply Theorem 1.2 from \cite{CLoptimalsphere}, which tells us that we can replace our sweepout
	$f$ of our Riemannian sphere $(S^2, g)$ of curves of length less than $L$ by a sweepout which contains only simple closed curves and constant curves.  Let this sweepout
	be parametrized by $[0,1]$, where $0$ and $1$ are mapped to the same constant curve.  Since it is smooth, we can find a finite number of subintervals $I_1, \dots, I_k$
	of $[0,1]$ such that the boundaries of $I_i$ are mapped to constant curves, the interior of $I_i$ is mapped to simple closed curves, and the degree of the map of $f$ restricted
	to $I_i$ is $d_i \neq 0$.  Furthermore, the sum of all of the degrees of these maps is equal to $1$, the degree of the map.  As a result, there is at least one such map that has odd degree.

	We now apply Proposition \ref*{Prop:exists} to this map to produce a zigzag which starts and ends at constant curves, and contains no other constant curves.  Since this zigzag is homotopic
	to the original map, it also has odd degree. Let $Z$ be a minimal zigzag of odd degree on the sphere which begins and ends at a constant curve, and only passes through simple closed curves, and which
	has minimal order.




	If $Z$ has order one, we are done. Otherwise, $Z$ has order at least $3$, and we first apply Proposition~\ref*{Prop:sphere}, and then Lemma \ref*{Lem:zigzag_modification} to $Z$; if we can divide it into two zigzags (the second
	possibility of the lemma), then the concatenation is homotopic to $Z$, and so has odd degree as a map, and so one of the zigzags 
	must have odd degree as a map but order smaller than $n$, which contradicts the minimality of the order of $Z$.  If we are in the first conclusion of Lemma~\ref*{Lem:zigzag_modification}, then the result is homotopic
	to $Z$, and so has odd degree as a map.  As in the proof of Theorem \ref*{thm:disc},
	$\gamma_{n-1}$ must be a constant curve, as it must be a minimizing geodesic in $A(\gamma_{n-1}, \gamma_n)$, and $\gamma_n$ is a constant curve with length $0$.  Thus, the degree of the last
	segment of $Z$ is $1$, which contradicts the minimality of
	the order of $Z$. This completes the proof.
\end{proof}

\newpage

\subsection*{Acknowledgements}

This work was partially supported by NSF grants CCF-1614562 and CCF-1054779, NWO project no. 639.023.208, ANR project ANR-16-CE40-0009-01 (GATO), NSERC Discovery Grant RGPIN 217655-13, and NSERC Postdoctoral Fellowship PDF-487617-2016.


\bibliographystyle{amsplain}
\bibliography{biblio}

\providecommand{\bysame}{\leavevmode\hbox to3em{\hrulefill}\thinspace}
\providecommand{\MR}{\relax\ifhmode\unskip\space\fi MR }
\providecommand{\MRhref}[2]{%
  \href{http://www.ams.org/mathscinet-getitem?mr=#1}{#2}
}
\providecommand{\href}[2]{#2}
\begin{thebibliography}{10}

\bibitem{Brightwell09submodularpercolation}
Graham~R. Brightwell and Peter Winkler, \emph{Submodular percolation}, SIAM J.
  Disc. Math \textbf{23} (2009), no.~3, 1149--1178.

\bibitem{esa-hh2017}
Benjamin Burton, Erin~Wolf Chambers, Marc van Kreveld, Wouter Meulemans, Tim
  Ophelders, and Bettina Speckmann, \emph{Computing optimal homotopies over a
  spiked plane with polygonal boundary}, Proceedings of the 25th Annual
  European Symposium on Algorithms, 2017.

\bibitem{homotopyheight}
Erin~W. Chambers and David Letscher, \emph{On the height of a homotopy},
  Proceedings of the 21st Canadian Conference on Computational Geometry, 2009,
  pp.~103--106.

\bibitem{CVE08}
Erin~Wolf Chambers, {\'E}ric {Colin de Verdi{\`e}re}, Jeff Erickson, Sylvain
  Lazard, Francis Lazarus, and Shripad Thite, \emph{Homotopic fr\'echet
  distance between curves or, walking your dog in the woods in polynomial
  time}, Computational Geometry \textbf{43} (2010), no.~3, 295 -- 311, Special
  Issue on 24th Annual Symposium on Computational Geometry (SoCG'08).

\bibitem{cmo-coh-17}
Erin~Wolf Chambers, Arnaud de~Mesmay, and Tim Ophelders, \emph{On the
  complexity of optimal homotopies}, In preparation, 2017.

\bibitem{cw2013}
Erin~Wolf Chambers and Yusu Wang, \emph{Measuring similarity between curves on
  2-manifolds via homotopy area}, Proc. 29th Ann. Symp. on CG, ACM, 2013,
  pp.~425--434.

\bibitem{cl-chidhh-14}
Gregory~R. Chambers and Yevgeny Liokumovich, \emph{Converting homotopies to
  isotopies and dividing homotopies in half in an effective way}, Geometric and
  Functional Analysis \textbf{24} (2014), no.~4, 1080--1100.

\bibitem{CLoptimalsphere}
\bysame, \emph{Optimal sweepouts of a {R}iemannian 2-sphere}, Journal of the
  European Mathematical Society (2014), to appear.

\bibitem{cr-mhcdrs-16}
Gregory~R. Chambers and Regina Rotman, \emph{Monotone homotopies and
  contracting discs on riemannian surfaces}, Journal of Topology and Analysis
  (2016).

\bibitem{hnss-hwdmml-16}
Sariel Har-Peled, Amir Nayyeri, Mohammad Salavatipour, and Anastasios
  Sidiropoulos, \emph{How to walk your dog in the mountains with no magic
  leash}, Discrete \& Computational Geometry \textbf{55} (2016), no.~1, 39--73.

\bibitem{NR1}
Alexander Nabutovsky and Regina Rotman, \emph{Linear bounds for lengths of
  geodesic loops on riemannian 2-spheres}, Journal of Differential Geometry
  \textbf{89} (2011), 217--232.

\bibitem{NR3}
\bysame, \emph{Length of geodesics and quantitative morse theory on loop
  spaces}, Geometric and Functional Analysis \textbf{23} (2013), 367--414.

\bibitem{NR2}
\bysame, \emph{Linear bounds for lengths of geodesic segments on riemannian
  2-spheres}, Journal of Topology and Analysis \textbf{5} (2013), 409--438.

\end{thebibliography}

\end{document}